\title{  The Kodaira dimension of contact 3-manifolds and geography of symplectic  fillings }

\documentclass[11pt]{article}

\usepackage{amsfonts}
\usepackage{amsthm}
\usepackage{xcolor}
\usepackage{graphicx}
\usepackage{amsmath}
\usepackage{enumerate}

\usepackage{amsmath,amsthm,amsfonts,amssymb,graphicx,psfrag}
\usepackage{color}
\usepackage[margin=1.5in]{geometry}

\newtheorem{thm}{Theorem}[section]

\newtheorem{prop}[thm]{Proposition}
\newtheorem{lemma}[thm]{Lemma}
\newtheorem{corr}[thm]{Corollary}

\theoremstyle{definition}
\newtheorem{rmk}[thm]{Remark}

\newtheorem{question}[thm]{Question}
\newtheorem{defn}[thm]{Definition}
\newtheorem{eg}[thm]{Example}

\begin{document}

\author{Tian-Jun Li and Cheuk Yu Mak\thanks{Both authors are supported by NSF-grant DMS 1207037.}}
\AtEndDocument{\bigskip{\footnotesize
  \textsc{School of Mathematics, University of Minnesota, Minneapolis, MN 55455} \par
  \textit{E-mail address}, Tian-Jun Li: \texttt{tjli@math.umn.edu} \par
  \addvspace{\medskipamount}
  \textsc{School of Mathematics, IAS, Princeton, NJ, 08540} \par
  \textit{E-mail address}, Cheuk Yu Mak: \texttt{cymak@math.ias.edu} \par
   \addvspace{\medskipamount}

}}

\date{\today}

\maketitle

\begin{abstract}

 We  introduce the Kodaira dimension of contact 3-manifolds and  establish some basic properties.
 In particular, contact 3-manifolds with distinct Kodaria dimensions behave differently when it comes to   the geography of various kinds of  fillings.
On the other hand,  we also prove that, given any contact 3-manifold,  there is a lower bound of  $2\chi+3\sigma$ for all its
minimal symplectic fillings.
This is motivated by  Stipsicz's result in \cite{Sti03} for Stein fillings.     Finally, we discuss various aspects of exact self cobordisms of fillable 3-manifolds.
\end{abstract}

\tableofcontents

\section{Introduction}

Understanding symplectic fillings of a given contact manifold $(Y,\xi)$ is one of the fundamental questions in contact/symplectic topology.
The ambitious goal is to classify all the Stein, exact,  or minimal strong symplectic fillings of a given contact manifold $(Y,\xi)$.
Although this  has been achieved for several families of contact 3-manifolds (\cite{Gromov}, \cite{Mc91}, \cite{Lis08}, etc), not much is known for a general contact manifold, even in dimension $3$.
A more realistic but still ambitious goal is to classify homology types of fillings for contact $3$-manifolds.
In this direction,  it was conjectured  by Stipsicz \cite{Sti03} that
$$\{(b_1(N),\chi(N),\sigma(N)| (N,\omega) \text{ is a Stein filling of } (Y,\xi))\}$$
is a finite set for any $(Y,\xi)$.
This is, however, disproved by Baykur and Van Horn-Morris in \cite{BaVHM15} (see also \cite{BaVHM16}, \cite{DKP15}).

In  \cite{LMY}, the authors together with Yasui discover that this conjecture (in fact, stronger version)  holds  when $(Y,\xi)$ admits a {\bf Calabi-Yau}  or a {\bf uniruled}/{\bf adjunction} cap (see Section \ref{s:Kodaira}). Uniruled caps and adjunction caps share all the properties when it comes to constraining fillings and we speculate that uniruled caps and adjunction caps actually coincide. For this reason, we  will overlook adjunction caps in this paper.
%In fact we prove more:  all  Betti numbers  of  exact fillings of a contact 3-manifold admitting a Calabi-Yau cap are bounded, and same is true for minimal  strong fillings
%of a contact 3-manifold admitting a uniruled/adjunction cap.
%As already mentioned, for many classes of contact manifolds,  there are stronger restrictions on $\chi$ and $\sigma$.
%In \cite{LMY} we introduced   uniruled caps and  Calabi-Yau caps to unify many results in this direction.
%If a contact $3$-manifold $(Y,\xi)$ admits a uniruled (resp. adjunction) cap, then we call it a uniruled (resp. adjunction) contact $3$-manifold.
We find it convenient to  introduce the notion of Kodaira dimension of contact 3-manifolds in terms of such caps.

\begin{defn}
For a closed co-oriented contact $3$-manifold $(Y,\xi)$, the Kodaira dimension is defined as follows.
\begin{displaymath}
   Kod(Y,\xi) = \left\{
     \begin{array}{llr}
       -\infty &  \text{ if it admits a uniruled cap }\\
       0 & \text{ if it admits a Calabi-Yau cap but admits no uniruled cap} \\
       1 & \text{ if it does not admit Calabi-Yau cap or uniruled cap}
     \end{array}
   \right.
\end{displaymath}
\end{defn}

%Even though every $(Y,\xi)$ admits a symplectic cap \cite{EtHo02}, many contact $3$-manifolds do not admit Calabi-Yau or uniruled/adjunction cap.

With  this notion, the results in \cite{LMY} mentioned above imply  that the symplectic filling version of Stipsicz's conjecture holds for any contact 3-manifold with $Kod=-\infty$, and the exact filling version of  Stipsicz's conjecture holds for any contact 3-manifold with $Kod=0$.
Moreover, explicit homological bounds can be explained in the case $Kod=-\infty$ given a uniruled cap, and in the case $Kod=0$ given a Calabi-Yau cap.

There are many contact 3-manifolds with $Kod =1$.
For Stein fillings of an arbitrary  contact 3-manifold, there is the beautiful result of Stipsicz:

\begin{thm}[\cite{Sti03}]\label{t:Stipsicz}
For any   contact $3$-manifold $(Y,\xi)$, the set
$$\{2\chi(N)+3\sigma(N) \in \mathbb{Z}| (N, \omega) \text{ is a Stein filling of } (Y,\xi)\}$$
is bounded from below.
\end{thm}

The is the best topological obstruction  in the literature which works for all contact 3-manifolds $(Y,\xi)$.
%Its proof  relies essentially on the Seiberg-Witten theory of K\"ahler surfaces, a tricky gluing argument
%and  the fact that any Stein domain can be embedded in a minimal K\"ahler surface of general type (\cite{LisMat97}).
It is natural to ask whether the analogous statement is true for exact fillings, or even for minimal strong fillings, which is a much weaker condition compared to Stein fillings (see \cite{Ghi05}, \cite{Wen13}).
We are able to  obtain an affirmative answer to this general question.
% via the existence of caps, Taubes' symplectic Seiberg-Witten theory,  and a generalization of Donaldson's construction of symplectic  hypersurfaces to symplectic manifolds with concave boundary.

\begin{thm}\label{t:charNumBound}
For any   contact $3$-manifold $(Y,\xi)$, the set
$$\{2\chi(N)+3\sigma(N) \in \mathbb{Z}| (N, \omega) \text{ is a minimal strong symplectic filling of } (Y,\xi)\}$$
is bounded from below.
Moreover, the lower bound can be explicitly calculated given a polarized  symplectic cap.
\end{thm}

Theorem \ref{t:charNumBound} together with the above results in \cite{LMY} for $Kod=-\infty$ and $Kod=0$ contact 3-manifolds
 provide a comprehensive geography picture  for various fillings of  contact 3-manifolds with a fixed Kodaira dimension.

Our approach is  based on the notion of a maximal surface in    \cite{LiZh11} and the notion of a Donaldson cap, which we introduce below.

\begin{defn}\label{d:maximalSurface}
%{\cite{LiZh11}}
Let $(X,\omega)$ be a closed symplectic four manifold and $D$ be a smooth symplectic surface in $X$.
Then $D$ is called {\bf maximal} if any symplectic exceptional class in $(X,\omega)$ pairs positively with $[D]$.
\end{defn}

%Notice that, for a concave symplectic pair $(P,\omega_P,\alpha_P)$, $[(\omega_P,\alpha_P)]$ is a relative de Rham class of $P$ (see Section 2 of \cite{LMY}).

%\begin{defn}
A concave symplectic pair is $(P, \omega_P, \alpha_P)$ where $(P, \omega_P)$ is a concave symplectic manifold
and $\alpha_P$ is a contact one form on $\partial P$ induced by an inward pointing Liouville vector field.
Such a pair is called of rational period if $\frac{1}{2\pi}[(\omega_P, \alpha_P)]\in H^2(P,\partial P;\mathbb{Q})$.
%\end{defn}

\begin{defn}  \label{d:polarized}
Let $(P, \omega_P, \alpha_P)$ be a  concave symplectic pair with rational period. A {\bf closed} symplectic hypersurface $D$ is called a Donaldson hypersurface of   $(P, \omega_P, \alpha_P)$  if
it is Lefschetz dual to an integral  multiple of $\frac{1}{2\pi}[(\omega_P,\alpha_P)]$.
We will often just say that $D$ is a Donaldson hypersurface of  $(P, \omega_P)$.

A symplectic cap is called a Donaldson cap if it admits a Donaldson hypersurface, and a Donaldson cap with a chosen Donaldson hypersurface is called a polarized cap.
\end{defn}

%Now, we are ready to introduce polarized  cap mentioned in Theorem \ref{t:charNumBound} as well as giving a proof of it.
We remark that polarized cap is a notion  inspired by \cite{Bir01} and \cite{Op13-1}.
%In applications, it is sometimes convenient to use the following version of Donaldson hypersurfaces.

%\begin{defn} \label{d:polarized}      Given a concave symplectic manifold $(P, \omega_P)$, a  closed symplectic hypersurface $D$ in $P$ is called a  Donaldson hypersurface
%for $(P, \omega_P)$ if it is a Donaldson hypersurface of  $(P, \omega_P, \alpha_P)$ for some contact form $\alpha_P$ induced by inward pointing Liouville vector field.
%Such a  triple $(P, \omega_P, D)$  is called a polarized concave symplectic manifold, or a polarized symplectic cap.
%\end{defn}

In Section \ref{ss:boundE} we show that the lower bound of $2\chi +3\sigma$ follow from properties of maximal surfaces and the existence of Donaldson  cap, and we will establish such an existence  using \cite{EtHo02} and \cite{LisMat}.

Along the way, we also prove that

\begin{thm}\label{t:universalMinimal}
For any co-oriented contact $3$-manifold $(Y,\xi)$, there exist a symplectic cap $(P,\omega_P)$ of $(Y,\xi)$ such that
for any minimal strong symplectic filling $(N,\omega_N)$ of $(Y,\xi)$, the glued symplectic manifold $(N \cup P,\omega)$ is minimal.
In particular,  any minimal convex symplectic 4-manifold embeds into a minimal closed symplectic 4-manifold.
\end{thm}

%It was shown  in \cite{LisMat97} that any Stein domain embeds into a minimal Kahler surface of general type (see also \cite{AkOz02}).
Notice that the  cap  in Theorem \ref{t:universalMinimal}  is universal in the sense that it works for any minimal filling of $(Y, \xi)$ (cf. \cite{AkOz02}).
%The notion of a polarized symplectic cap will be explained in Definition \ref{d:polarized}.

% construction of a closed Donaldson hypersurface in certain concave symplectic manifold, is  of independent interest.

%\begin{thm}\label{t:closedDonHyper}
% Any rational period concave symplectic pair  $(P,\omega_P, \alpha_P)$  admits a
% Donaldson hypersurface.
%\end{thm}

%One interesting feature of Theorem \ref{t:closedDonHyper} is that it can be used to find homologous but non-isotopic Donaldson hypersurfaces of high degree in many closed symplectic manifolds (Corollary \ref{c:nonIsotopic}).
%We also remark that Theorem \ref{t:closedDonHyper} is a generalization of the theorem of Auroux, Gayet and Mohsen (\cite{AGM01}, see also Section \ref{s:Donaldson}).
%As an immediate consequence, we have

%\begin{corr} \label{c:general symp}
%Any concave symplectic manifold $(P, \omega_P)$ admits a {\bf closed} symplectic hypersurface.
%\end{corr}

%\begin{question}
%Is every cap a Donaldson cap?
%\end{question}

%We were made aware by McLean an essential difficulty in confirming it using the $\eta-$transversality.

Interestingly,  Theorem \ref{t:charNumBound}    leads to a similar restriction for exact self cobordisms of fillable contact 3-manifolds.

\begin{corr} \label{c:charSelfExact}
Suppose $(Y, \xi)$ is a strongly fillable contact $3$-manifold.
Then the set
$$\{ 2\chi(W)+3\sigma(W) \in \mathbb{Z} | (W, \omega) \text{ is an exact cobordism from } (Y, \xi) \text{ to itself}\}$$
is bounded below by $0$.
In particular, if it is also bounded above, then the set is $\{0\}$.
\end{corr}

Notice that, the lower bound in Corollary \ref{c:charSelfExact}      can always be achieved by the trivial cobordism.
We remark that the strong fillability assumption in Corollary \ref{c:charSelfExact} is necessary.
A striking result by Eliashberg and Murphy \cite{EliMur15} implies that any almost complex  cobordism  between two smooth manifolds can be made into an exact cobordism for some overtwisted contact structures on the two boundary.
Together with a classical result by Van de Ven \cite{VdeV66}, one can easily construct an exact cobordism from $(S^3,\xi_{OT})$ to itself with very negative $2\chi+3\sigma$ (see Remark \ref{p:AlmostComplex}).

The paper is organized as follows.
In Section \ref{s:Kodaira},    we introduce the Kodaria dimension of  contact $3$-manifolds, discuss some of its properties and compute it for two important families of  fibred contact 3-manifolds.
In Section \ref{s:maximal},  we first  observe that for convex symplectic 4-manifolds,
the minimality defined in terms of symplectic -1 spheres is equivalent to the smooth minimality defined in terms of smooth -1 spheres.
Next we  provide
explicit  bounds on $2\chi+3\sigma$ for minimal strong fillings in terms of maximal surfaces. Theorem  \ref{t:charNumBound} is then proved after
establishing the existence of a Donaldson cap.
%In Section 3 we establish the existence of Donaldson hypersurfaces in the concave setting in arbitrary dimension.  With an eye toward   obtaining the bound on $2\chi+3\sigma$ more effectively and  improving it in certain situations, we
%generalize maximal surfaces to maximal divisors and generalize Opshein's result in \cite{Op13-2} to this setting.
Finally,  we apply the circle of ideas to study the monoid of exact self cobordisms of fillable contact 3-manifolds.

The authors are supported by NSF-grants DMS  1207037. We appreciate  discussions with Youlin Li,
Mark McLean, Jie Min and Chris Wendl.

\section{Kodaira dimension of contact $3$-manifolds}\label{s:Kodaira}

We first review symplectic manifolds with contact boundary and fixing the notations.
Every contact manifold $(Y,\xi)$ in this paper is assumed to be closed and co-oriented.

A strong symplectic filling $(N,\omega_N)$ of $(Y,\xi)$ is a symplectic manifold with boundary $\partial N$ such that near $\partial N$, there is a locally defined
Liouville vector field $V$ (i.e. $d\iota_V \omega_N=\omega_N$) pointing {\bf outward} along $\partial N$ so that
the induced contact structure $\xi_N$ on $\partial N$ makes $(\partial N,\xi_N)$ contactomorphic to $(Y,\xi)$.
When $V$ is chosen, we denote the induced contact $1$-form on $\partial N=Y$ as $\alpha_N$.
We also call a strong symplectic filling a symplectic filling, a filling or a convex symplectic manifold.
An exact symplectic filling $(N,\omega_N)$ of $(Y,\xi)$ is a strong symplectic filling such that the Liouville vector field $V$ can be extended to a globally defined vector field.
A Stein filling is a special kind of exact symplectic filling and we refer readers to \cite{Oz15} for more details.

The definition of a strong symplectic capping $(P,\omega_P)$ of $(Y,\xi)$ is similar to that of a strong symplectic filling.
The only modification is that the locally defined Liouville vector field $V$ is required to point {\bf inward} instead of outward.
The induced contact $1$-form, which we call the {\bf Liouville $1$-form}, is denoted as $\alpha_P$ and $(P,\omega_P,\alpha_P)$ forms a concave symplectic pair.
A strong symplectic capping $(P,\omega_P)$ is also called a symplectic cap, a cap or a concave symplectic manifold.

For a choice of a filling $(N,\omega_N)$ and a cap $(P,\omega_P)$ of $(Y,\xi)$, we can glue them together to get a closed symplectic manifold $X=N \cup_Y P$ by inserting part of the symplectization of $(Y,\xi)$ (see \cite{Et98}).
The symplectic form $\omega$ on $X$ is not canonical and our convention is that $\omega|_P=\omega_P$ so $(N,\omega|_N)$ is obtained by attaching part of the symplectization of $(Y,\xi)$ to $(N,\lambda_P\omega_N)$, for some $\lambda_P >0$ small.
The actual rescaling factor $\lambda_P$ is not very important to our discussion so whenever a filling is glued with a cap, $\lambda_P$ is chosen implicitly.

An oriented cobordism $W$ from a closed oriented manifold $Y_-$ to  another $Y_+$
is a compact oriented manifold $W$ such that $\partial W= Y_+ \cup (-Y_-)$.  An oriented cobordism $W$ with a symplectic structure $\omega$ compatible with the orientation is called a (strong) symplectic cobordism if $(W, \omega)$ is symplectic concave
at $Y_-$ and convex at $Y_+$. A symplectic cobordism $(W, \omega)$ is called exact if the primitive
1-form near $Y_{\pm}$ extends to a global primitive 1-form.
Symplectic gluing can also be performed between symplectic cobordism and symplectic filling/cap.
These notations and conventions are used throughout the whole paper.

In Section \ref{ss:uniruledCY}, we recall  Kodaira dimension of closed symplectic 4-manifolds and various caps introduced in \cite{LMY}. Then  we introduce the Kodaria dimension of closed co-oriented contact $3$-manifolds which captures the topological complexity of its symplectic fillings.
Finally, we provide examples of Stein fillable contact torus bundles and circle bundles which attain various contact Kodaira dimensions.

\subsection{Symplectic Kodaira dimension, uniruled  and Calabi-Yau caps}\label{ss:uniruledCY}

Let $X$ be a closed, oriented smooth 4-manifold.
Let ${\mathcal E}_X$ be the set of cohomology
classes whose Poincar\'e duals are represented by smoothly embedded
spheres of self-intersection $-1$. $X$ is said to be (smoothly)
minimal if ${\mathcal E}_X$ is the empty set.
% \end{definition}
Equivalently, $X$ is minimal if it is not the connected sum of
another manifold  with $\overline{\mathbb {CP}^2}$.

\medskip
Suppose $\omega$ is a symplectic form compatible with the orientation.
$(X,\omega)$  is said to be (symplectically) minimal if ${\mathcal
E}_{\omega}$ is empty, where
$${\mathcal E}_{\omega}=\{E\in {\mathcal
E}_X|\hbox{ $E$ is represented by an embedded $\omega-$symplectic
sphere}\}.$$
We say that $(Z, \tau)$
is a minimal model of $(X, \omega)$ if $(Z, \tau)$ is  minimal and $(X, \omega)$ is
a symplectic blow up of $(Z, \tau)$.
A basic fact proved using Taubes' SW theory (\cite{Taubes96}, \cite{LiLiu95}, \cite{Li99})  is:
  ${\mathcal E}_{\omega}$ is
empty if and only if ${\mathcal E}_X$ is empty. In other words, $(X,
\omega)$ is symplectically minimal if and only if $X$ is smoothly
minimal.

%\begin{definition}\label{sym Kod'}
For minimal $(X,\omega)$,  the Kodaira dimension of
$(X,\omega)$
 is defined in the following way in \cite{Li06s} (see also \cite{McSa96} \cite{Le97}):

$$
\kappa^s(X,\omega)=\begin{cases} \begin{array}{lll}
-\infty & \hbox{ if $K_{\omega}\cdot [\omega]<0$ or} & K_{\omega}\cdot K_{\omega}<0,\\
0& \hbox{ if $K_{\omega}\cdot [\omega]=0$ and} & K_{\omega}\cdot K_{\omega}=0,\\
1& \hbox{ if $K_{\omega}\cdot [\omega]> 0$ and} & K_{\omega}\cdot K_{\omega}=0,\\
2& \hbox{ if $K_{\omega}\cdot [\omega]>0$ and} & K_{\omega}\cdot K_{\omega}>0.\\
\end{array}
\end{cases}
$$
%\end{definition}
Here $K_{\omega}$ is defined as the first Chern class of the
cotangent bundle for any almost complex structure compatible with
$\omega$.

The invariant $\kappa^s$ is well defined  since there does not
exist a minimal $(X, \omega)$ with $$K_{\omega}\cdot [\omega]=0,
\quad\hbox{and}\quad  K_{\omega}\cdot K_{\omega}>0.$$
This again follows from Taubes' SW theory \cite{Li06s}.
Moreover,
 $\kappa^s$ is independent of $\omega$, so it is an oriented diffeomorphism invariant of $X$.

 The Kodaira dimension of a non-minimal manifold is defined to be
that of any of its minimal models.
This definition is well-defined and independent of choice of minimal model so $\kappa^s(X, \omega)$ is well-defined for any $(X, \omega)$ (cf. \cite{Li06s}).
When $\kappa^s(X, \omega)=0$ (resp. $\kappa^s(X, \omega)= -\infty$), $(X,\omega)$ a called a {\bf symplectic Calabi-Yau manifold} (resp. {\bf symplectic uniruled manifold})

In \cite{LMY} we introduced the analogues  of symplectic uniruled and Calabi-Yau manifolds for concave  4-manifolds.
%\subsection{Uniruled, adjunction,  and Calabi-Yau caps}

\begin{defn}\label{d:CY}

%\end{defn}

%\begin{defn}\label{d:uniruled}
 A {\bf uniruled cap} $(P,\omega_P,\alpha_P)$ of a contact 3-manifold $(Y,\xi)$ is a concave symplectic pair such that $c_{1}(P)\cdot[(\omega_P,\alpha_P)]>0$.
 %for some Liouville one form $\alpha_P$.

 % An {\bf adjunction cap} $(P,\omega_P)$ of a contact 3-manifold $(Y,\xi)$ is a symplectic cap
 % such that there exist a smoothly embedded (not necessarily symplectic) genus $g$ surface $S$ in $P$ with
 % $[S]^2 \ge \max \{2g-1,0\}$.
 % If $[S]^2=0$, we further require that $[S] \in H_2(P,\mathbb{Q})$ does not lie in the image of $H_2(Y,\mathbb{Q})$ under the natural map induced by inclusion.

 A {\bf Calabi-Yau cap} $(P,\omega_P)$ of a contact 3-manifold $(Y,\xi)$ is a symplectic cap such that
 $c_1(P)$ is torsion.
\end{defn}

%It is speculated in \cite{LMY} that uniruled caps and adjunction caps are in fact the same notion, possibly after a symplectic deformation.

For a concave symplectic pairs $(P,\omega_P,\alpha_P)$, we have the following possibilities:
\begin{itemize}
\item $c_1(P) \cdot [(\omega_P,\alpha_P)] > 0$: $(P,\omega_P,\alpha_P)$ is a uniruled cap.
\item $c_1(P) \cdot [(\omega_P,\alpha_P)] =0$: $(P,\omega_P,\alpha_P)$ is either a Calabi-Yau cap or a uniruled cap after symplectic deformation.
%The first of which is a uniruled cap and the second of which is either a Calabi-Yau cap or a uniruled cap after symplectic deformation.
%The last one is a bit subtle.
\item $c_1(P) \cdot [(\omega_P,\alpha_P)] < 0$: it is still possible for $(P,\omega_P,\alpha_P)$ to become a uniruled cap after a symplectic deformation.
\end{itemize}
The second bullet above is true because if $c_1(P)$ is not torsion, we can find a $2$-form $\beta$ such that $c_1(P) \cdot [\beta]>0$ and deform
$(P,\omega_P,\alpha_P)$ in the $\beta$-direction to make it uniruled after deformation.

Uniruled caps and Calabi-Yau caps give surprisingly good control to symplectic fillings of $(Y,\xi)$. To state such results,
it is convenient to introduce the following notions for fillings.

\begin{defn}
For a symplectic cap $(P,\omega_P)$ and a symplectic filling $(N,\omega_N)$ of $(Y,\xi)$, the $(P,\omega_P)$-Kodaira dimension of $(N,\omega_N)$, $\kappa^{(P,\omega_P)}(N,\omega_N)$, is defined to be
the Kodaira dimension of $(X=N \cup_Y P,\omega)$.
\end{defn}

When $\kappa^{(P,\omega_P)}(N,\omega_N)=0$ (resp. $\kappa^{(P,\omega_P)}(N,\omega_N)=-\infty$), $(N,\omega_N)$ is called $(P,\omega_P)$-Calabi-Yau (resp. $(P,\omega_P)$-uniruled).

\begin{thm}[\cite{LMY}]\label{t:uniruled}
 Suppose  $Kod(Y,\xi)=-\infty$  and  $(P,\omega_P)$ is  a uniruled cap of $(Y, \xi)$.
Then $\kappa^{(P,\omega_P)}(N,\omega_N)=-\infty$ and
 there are integers $n_1,n_2,n_3$ depending only on $(P, \omega_P)$ such that $b_i(N) \le n_i$ for $i=1,2,3$ for any minimal symplectic filling $(N,\omega_N)$ of $(Y,\xi)$. Moreover, $b_2^+(N)=0$.
\end{thm}

\begin{thm}[\cite{LMY}]\label{t:CY}
Suppose $Kod(Y, \xi)=0$ and  $(P,\omega_P)$ is  a Calabi-Yau cap of $(Y, \xi)$.
Then $\kappa^{(P,\omega_P)}(N,\omega_N) \in \{ -\infty,0 \}$ and
there are integers $n_1,n_2,n_3$ depending only on $(P, \omega_P)$ such that $b_i(N) \le n_i$ for $i=1,2,3$ for any exact filling $(N,\omega_N)$ of $(Y,\xi)$.
\end{thm}

%It is tempting to categorize concave symplectic pairs $(P,\omega_P,\alpha_P)$ into three types,
%namely, $c_1(P) \cdot [(\omega_P,\alpha_P)] > 0$, $c_1(P) \cdot [(\omega_P,\alpha_P)] =0$ and $c_1(P) \cdot [(\omega_P,\alpha_P)] < 0$.
%The first of which is a uniruled cap and the second of which is either a Calabi-Yau cap or a uniruled cap after symplectic deformation.
%The last one is a bit subtle.
%Instead, we categorize contact $3$-manifolds into three types.

\subsection{Properties of $Kod(Y, \xi)$}

Recall that, for a contact 3-manifold $(Y, \xi)$,  we have defined $Kod(Y, \xi)=-\infty$ if it admits a uniruled cap, $Kod(Y, \xi)=0$ if it does not admit a uniruled cap
but admits a Calabi-Yau cap, and $Kod(Y, \xi)=1$ if it does not admit either a uniruled cap or a Calabi-Yau cap.
The followings are either contained in \cite{LMY} or immediate from the definition.

\begin{lemma}\label{l:KodcobIncrease}

The contact Kodaira dimension has the following properties.
\begin{itemize}
\item $Kod (Y, \xi)=-\infty$ if $(Y, \xi)$ is overtwisted.

\item If the positive end of a strong symplectic cobordism has $Kod=-\infty$, then so does the negative end.

\item If $(Y, \xi)$ is co-fillable, ie. it is a connected component of the boundary of a strong semi-filling with disconnected boundary, then $Kod(Y, \xi)\geq 0$.

\item If $c_1(Y, \xi)$ is not torsion, then $Kod(Y, \xi)\ne 0$.
\end{itemize}
\end{lemma}

\begin{proof}
The first three bullets corresponds to Lemma 4.13, Lemma 4.6 and  Corollary 4.8 in \cite{LMY}.
The last bullet is immediate.
% It is proved in \cite{LMY} that if $(Y_+,\xi_+)$ is uniruled then so is $(Y_-,\xi_-)$.
% Hence, $Kod(Y_+,\xi_+)=-\infty$ implies $Kod(Y_-,\xi_-)=-\infty$.

%For the third bullet, capping all other positive ends of the semi-filling we obtain a filling of $(Y, \xi)$. If we choose the cap to have %$b_2^+\geq 2$, which always exist, then
%after possibly blowing down the exceptional spheres, we get a minimal filling of $(Y, \xi)$ with $b_2^+\geq 2$.

\end{proof}

Since the unit cotangent bundle  of surface with high genus is a maximal element with respect to symplectic cobordism (because it is co-fillable) and it has  $Kod =0$ (see \cite{LMY}), we generally do not have the inequality in the second bullet for a strong symplectic cobordism.

In light of the bullets $1$ and $2$, we ask

\begin{question} Suppose $(Y, \xi)$ is a contact 3-manifold.
\begin{itemize}
\item If $(Y, \xi)$ is not fillable,  do we have $Kod(Y, \xi)=-\infty$?

\item If  $(Y, \xi)$ is fillable, do we have  $Kod(Y, \xi)\geq \min_{(X, \omega)} \kappa^s(X, \omega)$, where $(X, \omega)$ is a closed symplectic 4-manifold
containing $(Y, \xi)$ as a separating hypersurface?

\end{itemize}

\end{question}

Regarding the first bullet, a related question that was asked by Wendl is whether all non-fillable contact manifold is cobordant to an overtwisted contact manifold.
The answer is negative since there are   non-fillable $3$-manifolds  with non-trivial contact invariant \cite{LiSt04}.
However, Wendl proved that any contact $3$ manifold with planar torsion is symplectic cobordant to an overtwisted contact manifold, and hence has negative Kodaira dimension \cite{Wen13}.
Therefore, our question is a weaker version of Wendl's question.

Regarding the second bullet, it is true when $Kod(Y, \xi) = -\infty$.
When $Kod (Y, \xi) =0$, there is a subtlety  if $(Y, \xi)$ is not exactly fillable.
 When $Kod (Y, \xi)=1$,
the proposed inequality says  that  there exists  a contact embedding  in  $(X, \omega)$ with $\kappa(X, \omega)$ at most  $1$.
Notice that  there are abundant  $(X, \omega)$ with $\kappa^s=1$  due to Gompf.
Moreover, since many $\kappa^s=1$ symplectic manifold have tori fibration, we speculate that whether any
fillable contact manifold can be embedded in a $\kappa=1$ symplectic manifold is related to the question of supporting
genus of the contact manifold.

We end this subsection with the following observation.

\begin{prop}
 If $(W,\omega_W,\alpha_W)$ is an exact cobordism with negative end $(Y_-,\xi_-)$ and positive end $(Y_+,\xi_+)$,
 then $$Kod(Y_-,\xi_-) \le Kod(Y_+,\xi_+).$$

\end{prop}

\begin{proof}

We can  assume $Kod(Y_+,\xi_+)=0$ by the second bullet of  Lemma \ref{l:KodcobIncrease}.
 Let $(P,\omega_P)$ be a Calabi-Yau cap of $(Y_+,\xi_+)$.
 Then $(P \cup_{Y_+} W,\omega)$ is a symplectic cap of $(Y_-,\xi_-)$. Since   $(W,\omega_W,\alpha_W)$ is exact and  $(P,\omega_P)$ is Calabi-Yau, we have
  $c_1(P \cup_{Y_+} W) \cdot [(\omega,\alpha_W|_{Y_-})]=0$ (see section $2$ of \cite{LMY}).
 By the observation right after Definition \ref{d:CY}, the cap $P \cup_{Y_+} W$  is either uniruled (after a symplectic deformation) or Calabi-Yau and hence the inequality.

  \end{proof}

\begin{rmk}
Zhang introduced the topological Kodaira dimension for closed 3-manifolds $\kappa(Y)$ in \cite{Zh14} using Thurston's 8 geometries and Perelman's solution to  the geometrization conjecture. A major property is  that $\kappa(Y)\geq \kappa(Y')$ if there is a nonzero degree map from $Y$ to $Y'$.
%We speculate that $\kappa(Y)\geq \min_{\xi} Kod(Y, \xi)$.
However, it is not clear that how $\kappa(Y)$ and $Kod(Y, \xi)$ are related.
For example, $Kod(Y, \xi_{OT})=-\infty$ for any overtwisted contact structure $\xi_{OT}$ (Lemma \ref{l:KodcobIncrease}).
There are also Stein fillable $(Y, \xi)$ with  $\kappa(Y)=1$ and  $Kod(Y, \xi)=-\infty$ (Lemma \ref{l:sharp}).

One possible relation is that $\kappa(Y)\geq  \min_{\xi}Kod(Y, \xi)$. This is subtle when $\kappa(Y)=0$.
In the example below we show that there are also some torus bundle $Y$ such that $\kappa(Y)=0$ and $Kod(Y, \xi)=1$ for
some Stein fillable contact structure $\xi$.

\end{rmk}

\subsection{Contact Kodaira dimensions for some fibered manifolds}

We discuss the contact Kodaira dimension for various torus bundles over circle.
In particular, we provide Stein fillable torus bundles with $Kod(Y,\xi)=1$.

There are many torus bundles with $Kod(Y,\xi)=-\infty$.
A detailed study can be found in \cite{GoLi14}.
A common feature of these examples is that they can be realized as the contact boundary of a divisor cap with the symplectic divisor being a cycle of
symplectic spheres or a single symplectic torus.

To give torus bundles with $Kod(Y,\xi)=0,1$, we need to first review some basic facts about cusp singularities \cite{Lo81}.
Consider an isolated cusp normal surface singularity.
Its minimal resolution is a cycle of rational curves with negative definite self-intersection form $Q_{\Gamma}$.
We denote the cycle of rational curves as $D$ and regard it as a symplectic divisor.
The boundary of a regular neighborhood $N(D)$ of $D$ is a torus bundle over circle $Y$ \cite{Ne81}.
Moreover, $Y$ is equipped with the canonical contact structure $\xi$ as a link of complex isolated normal surface singularity.
This setup should be viewed as the dual of the one in \cite{GoLi14}.

$D$ is called {\it embeddable} if there is a smooth (complex) rational surface with an effective reduced anticanonical divisor whose dual graph is the same as that of $D$.
$D$ is called {\it smoothable} if the cusp singularity associated to $D$ is smoothable as a complex surface singularity.
When $D$ is smoothable, the smoothings are diffeomorphic to the complements of $\check{D}$ in $X$, where $\check{D}$ is the resolution divisor of the dual cusp singularity associated to $D$
and $X$ is a smooth rational surface so that $\check{D}$ can be embedded into.
In particular, smoothings of cusps have $b_2^+=1$.
We can now completely determine when a cusp is smoothable/embeddable thanks to Hirzebruch-Zagier algorithm which checks embeddability combinatorically, and proofs of Looijenga conjecture which relates smoothablility of a cusp with embeddability of the dual cusp \cite{GrHaKe11}, \cite{En15}.

\begin{lemma}\label{l:Kod0torusBundle}
When a cusp is embeddable and smoothable, then its link $(Y,\xi)$ has $Kod(Y,\xi)=0$.
\end{lemma}

\begin{proof}
When a cusp is embeddable, say $D$ is embedded into $X$, then the complement $P$ of the regular neighborhood $N(D)$ of $D$ in $X$ is a Calabi-Yau cap of $(Y,\xi)$.
In this case, $Kod(Y,\xi) \le 0$.
Moreover, $D$ being smoothable implies that $(Y,\xi)$ admits a Stein filling (its smoothings) with $b_2^+=1$ so $(Y,\xi)$ cannot admit uniruled cap and $Kod(Y,\xi)=0$.
\end{proof}

One can use this explicit Calabi-Yau cap to study exact fillings of $(Y,\xi)$ by Theorem \ref{t:CY} but we do not pursue it here.
We remark that one can check explicitly there are many embeddable and smoothable $D$.
On the other hand, there are also many smoothable but not embeddable cusp singularities.

\begin{lemma}
When a cusp is not embeddable but smoothable, then its link $(Y,\xi)$ has $Kod(Y,\xi)=1$
\end{lemma}

\begin{proof}
As in the proof of Lemma \ref{l:Kod0torusBundle}, $D$ being smoothable implies that $Kod(Y,\xi)\neq -\infty$.
Suppose $(Y,\xi)$ admits a CY cap $(P,\omega_P)$, then $(X=P \cup N(D),\omega)$ is either a minimal Calabi-Yau manifold or a non-minimal uniruled manifold \cite{LMY}.
Since $c_1(D) \neq 0$, $X=P \cup N(D)$ is a non-minimal uniruled manifold.
Moreover, $c_1|_P$ is torsion and the Poincare dual of $c_1|_{N(D)}$ is represented by $D$ (this can be checked by adjunction)
implies that $D$ represents the Poincare dual of the first Chern class in $(X,\omega)$ and hence a symplectic Looijenga pair in the sense of \cite{LM15}.
By the classification result in \cite{LM15}, the embeddability of $D$ as a symplectic divisor in a symplectic rational manifold is the same as that in the complex (K\"ahler) case.
By the assumption that $D$ is not embeddable in the complex sense, we get a contradiction.
As a result,  $Kod(Y,\xi)=1$.
\end{proof}

The following lemma provides another illustrative family of examples of contact manifolds with different Kodaira dimensions.

\begin{lemma}\label{l:sharp}
Let $(Y_{g,n},\xi_{g,n})$ be the boundary of a neighborhood of a symplectic surface $D_{g,n}$ with genus $g$ and self-intersection number $n >0$ equipped with the canonical contact structure.
Then,
\begin{displaymath}
   Kod(Y_{g,n},\xi_{g,n}) = \left\{
     \begin{array}{llr}
       -\infty & n>2g-2\\
       0 & n=2g-2 \\
       1 & n < 2g-2
     \end{array}
   \right.
\end{displaymath}

\end{lemma}

\begin{proof}
As a neighborhood of  $D_{g,n}$, $(P_{g,n},\omega_{P_{g,n}})$  is a symplectic cap of $(Y_{g,n},\xi_{g,n})$.
 When $n>2g-2$, $(P_{g,n},\omega_{P_{g,n}})$ is a uniruled cap and hence $Kod(Y_{g,n},\xi_{g,n})=-\infty$.
 When $n=2g-2$, $(P_{g,n},\omega_{P_{g,n}})$ is a Calabi-Yau cap.
 However, $(Y_{g,n},\xi_{g,n})$ does not admit a uniruled cap because it is co-fillable.
  When $n<2g-2$, it admits no Calabi-Yau cap because $c_1(\xi_{g,n})$ is not torsion.
 To show that it admits no uniruled cap, we consider the Calabi-Yau cap  $(P_{g,2g-2},\omega_{P_{g,2g-2}})$ with $D_{g,2g-2}$ inside.
 We can do $2g-2-n$ small symplectic blowups along $D_{g,2g-2}$ and result in a symplectic surface $D'$ of genus $g$ and self intersection number $n$.
 The neighborhood of $D'$ is symplectic deformation equivalent to $(P_{g,n},\omega_{P_{g,n}})$.
 By deleting the neighborhood of $D'$, we get a symplectic cobordism with negative end being $(Y_{g,2g-2},\xi_{g,2g-2})$ and positive end being $(Y_{g,n},\xi_{g,n})$.
 Since $(Y_{g,2g-2},\xi_{g,2g-2})$ is co-fillable, so is $(Y_{g,n},\xi_{g,n})$.
 Hence $Kod(Y_{g,n},\xi_{g,n})=1$.

\end{proof}

%\begin{rmk}
% By Corollary \ref{c:exactCob}, any contact $3$ manifold $(Y,\xi)$ is exact cobordant to one of the $(Y_{g,n},\xi_{g,n})$ in Lemma \ref{l:sharp}.
% Lemma \ref{l:KodcobIncrease} provides an inequality of Kodaira dimension under exact cobordism.
% Combining these results, it is not hard to see that $$Kod(Y,\xi)=\min \{Kod(Y_{g,n},\xi_{g,n})| \text{ $(Y,\xi)$ is exact cobordant to $(Y_{g,n},\xi_{g,n})$}\}$$
%\end{rmk}

\section{Maximal  surface and the lower bound on $2\chi+3\sigma$}\label{s:maximal}

\subsection{Equivalence of minimality for convex 4-manifolds}

In this subsection, we make an observation on  symplectic exceptional spheres in convex 4-manifolds.

Let $(N,\omega_N)$ be a convex symplectic manifold.
It is called smoothly minimal if there is no smoothly embedded sphere of self-intersection $-1$.
It is called (symplectically) minimal if there is no symplectically embedded sphere of self-intersection $-1$.
For homological reason, any exact/Stein filling is minimal.
The following proposition shows that these two are in fact the same notion, similar to the situation for closed symplectic four manifolds.

\begin{prop} \label{p:minimality}
A convex symplectic 4-manifold   $(N,\omega_N)$ is symplectically minimal if and only if  it is smoothly minimal.
\end{prop}

We first recall a result in \cite{LMY}.
%Notice that, for a concave symplectic pair $(P,\omega_P,\alpha_P)$, $[(\omega_P,\alpha_P)]$ is a relative de Rham class of $P$ (see Section 2 of \cite{LMY}).

\begin{lemma}[\cite{LMY}, Corollary 2.12]\label{l:maximalClass}
 Let $(P,\omega_P,\alpha_P)$ be a concave symplectic pair and $(N,\omega_N)$ a symplectic filling of $(Y,\xi)$.
 Then any symplectic exceptional class in $(N \cup_Y P,\omega)$ which admits no embedded symplectic representative in $(N,\omega_N)$ pairs positively with $PD[(\omega_P,\alpha_P)]$

 In particular, if $(N,\omega_N)$ is (symplectically) minimal, any symplectic exceptional class in $(N \cup_Y P,\omega)$ pairs positively with $PD[(\omega_P,\alpha_P)]$.
\end{lemma}

\begin{proof}[Proof of Proposition \ref{p:minimality}]
Clearly $N$ is symplectically minimal if it is smoothly minimal.
To prove the converse, we will assume that $N$ is not smoothly minimal and show that it cannot be symlectically minimal.
Let $e\in H_2(N)$ be the class of a smoothly embedded $-1$ sphere in $N$.

We glue  $N$ along its contact boundary $(Y, \xi)$  with a concave 4-manifold  $(P, \omega_P)$  to obtain a closed symplectic 4-manifold $(X, \omega)$.
Further, we can assume that $b^+_2(P)>1$ (by \cite{EtHo02}).
Denote still by  $e$ the image of $e$ under the natural map $H_2(N) \to H_2(X)$.   Notice that  $b^+_2(X)>1$ since $b^+_2(P)>1$. By a result of Taubes \cite{Taubes96},
there is  a $\omega-$symplectic $-1$ sphere $S$   in the class $e$ or $-e$.

%Since $(P, \omega_P)$ is rational, there is  a Donaldson hyper surface $D$ in $P$.

By Lemma \ref{l:maximalClass}, $S$ pairs positively with $PD[(\omega_P,\alpha_P)] \in H_2(X)$.
However, $PD[(\omega_P,\alpha_P)]$ comes from the natural map $H_2(P) \to H_2(X)$.
This contradicts to the fact that $e$ is represented by a smooth sphere in $N$.
\end{proof}

By the same argument, we can also obtain the following consequences.

\begin{corr}\label{c:disjointExceptional}
Let $(N,\omega_N)$ be a convex symplectic manifold.
If there is a smooth $-1$ sphere in $N$, there is a symplectic $-1$ sphere homologous to it up to sign.

Moreover, the   classes of symplectic $-1$ spheres are pairwise orthogonal.
\end{corr}

It turns out that this corollary will be essential when we study self exact cobordism in Section \ref{s:ExactCob}.
A natural  question is whether the corresponding result of Proposition \ref{p:minimality} is true for concave symplectic 4-manifolds.
We remark that removing a ball in a rational 4-manifold with more than two blow-ups
gives a counterexample of the corresponding result of Corollary \ref{c:disjointExceptional} for concave symplectic 4-manifolds.

\subsection{Bounding exceptional curves by a maximal surface}\label{ss:boundE}

We now discuss how to use a  maximal symplectic surface (see Definition \ref{d:maximalSurface}) to bound the number of  exceptional curves in a closed symplectic four manifold.

\begin{lemma}\label{l:non-uniruled_maximal}
Let $(X,\omega)$ be a non-uniruled closed symplectic four manifold.
If $D$ is maximal, then $c_1(X,\omega)^2 \ge c_1(X,\omega) \cdot D$.
In particular, the number of exceptional spheres is bounded above by $K_{\omega} \cdot D$.
\end{lemma}

\begin{proof}
 By Taubes, any  minimal model $(\widetilde X, \widetilde \omega)$  of $(X, \omega)$ has $c_1\cdot c_1\geq 0$ and all the exceptional classes are pairwise orthogonal.
 Let $L$ be the number of exceptional classes. Then
 $$c_1(X, \omega)^2=c_1(\widetilde X, \widetilde \omega)^2-L\geq -L.$$

 Let $e_1, ..., e_L$ be the exceptional classes. Then by Taubes,
 $K_{\widetilde \omega}=K_{\omega}-\sum e_i$ is a GT class and hence  represented by a $J-$holomorphic curve for any
 $\widetilde \omega-$compatible $J$.  We pick such a $J$ so that $D$ is $J-$holomorphic.
 Since $D$ has positive self-intersection, by positivity of intersection, we have
 $D\cdot K_{\widetilde \omega}\geq 0$.
 Therefore we have
 $$L\leq (\sum e_i \cdot D)=(K_{\omega}-K_{\widetilde \omega})\cdot D\leq K_{\omega}\cdot D=2g(D)-2-D\cdot D,$$

\end{proof}

\begin{lemma}\label{l:uniruled_maximal}
Let $(X,\omega)$ be a uniruled closed symplectic four manifold.
If $D$ is maximal surface of genus $g >0$, then
$$c_1(X,\omega)^2 \ge c_1(X) \cdot D +2-2g$$
if $D$ is not a section
$$c_1(X,\omega)^2=8-8g$$
if $D$ is a section.
\end{lemma}

\begin{proof}
Since $D$ is not a sphere, $D$ is a maximal surface in the sense of \cite{LiZh11} and hence   satisfies
$(-c_1(X)+[D])^2\geq 0$ when $D$ is not a section.
By adjunction, we have $c_1(X) \cdot D=D \cdot D +2-2g$.
Hence $c_1(X,\omega)^2 \ge c_1(X) \cdot D +2-2g$.

When $D$ is a section, it is maximal only when $(X,\omega)$ is minimal.
In this case, $(X,\omega)$ is a sphere bundle over a genus $g$ surface and hence  $c_1(X,\omega)^2=(2\chi+3\sigma)(X)=8-8g$.

\end{proof}

\begin{corr}\label{c:explicitBounds}
 Let $D$ be a maximal symplectic surface with genus $g>0$ in a concave symplectic manifold $(P,\omega)$.
 There is  a lower bound   on    $(2\chi+3\sigma)(N)$  of any  minimal strong symplectic filling $N$ of $Y=\partial P$ given by
$$(2\chi+3\sigma)(N) \ge \min\{c_1(P) \cdot D +2-2g,8-8g\}-(2\chi+3\sigma)(P)$$
Moreover, if  $\kappa^{(P,\omega_P)}(N,\omega_N) \ge 0$, we have
         $$(2\chi+3\sigma)(N) \ge c_1(P) \cdot D - (2\chi+3\sigma)(P)$$
\end{corr}

\begin{proof}
  Notice that, for any minimal strong symplectic filling $N$ of $Y$, the glued symplectic manifold $X=N \cup_Y P$ satisfies
 $c_1(X, \omega_X)^2=2\chi(X)+3\sigma(X)$, and
  $$\sigma(X)=\sigma(P)+\sigma(N), \quad  \chi(X)=\chi(P)+\chi(N).$$
  so it suffices to prove that
  $$(2\chi+3\sigma)(X) \ge c_1(P) \cdot D +2-2g$$ in general and
  $$ (2\chi+3\sigma)(X) \ge c_1(P) \cdot D$$ when $\kappa^{(P,\omega_P)}(N,\omega_N) \ge 0$, which in turn follows from Lemma \ref{l:non-uniruled_maximal} and Lemma \ref{l:uniruled_maximal}.
\end{proof}

Note that $b_2^+(P)+b_2^+(N)>1$ is sufficient to guarantee that $\kappa^{(P,\omega_P)}(N,\omega_N) \ge 0$.

Moreover, when $n \le 2g-2$, the bound obtained for $\kappa^{(P_{g,n},\omega_{P_{g,n}})}(N,\omega_N) \ge 0$ in Corollary \ref{c:explicitBounds} is sharp, where $(P_{g,n},\omega_{P_{g,n}})$ is the neighborhood of $D_{g,n}$.
 By tracing the way we obtain the bound in Corollary \ref{c:explicitBounds}, the bound for $\kappa^{(P_{g,n},\omega_{P_{g,n}})}(N,\omega_N) \ge 0$ is sharp
 if there is a minimal symplectic filling $(N,\omega_N)$ such that $\kappa(N \cup_Y P,\omega)=0$ and every exceptional class $e$ in $(N \cup_Y P,\omega)$
 satisfy $e \cdot D_{g,n}=1$.
 Notice that, there are symplectic surfaces of any genus in a $K3$ surface.
 Similar to above, by performing $2g-2-n$ blow-ups along a symplectic surface of genus $g$, we get a symplectic surface $\overline{D}$ of genus $g$ and self intersection number $n$.
 The complement of a neighborhood of $\overline{D}$ realizes the lower bound obtained in  Corollary \ref{c:explicitBounds}.

\subsection{Existence of maximal cap}

\begin{defn}
 Let $(P,\omega_P)$ be a concave symplectic manifold and $D$ be a smooth symplectic surface in $P$.
Then $D$ is called {\bf maximal} if, for any minimal symplectic filling $(N,\omega_N)$ of $\partial P$, $D$ is maximal in $(N \cup_{\partial P} P,\omega)$.

A cap is called maximal if it admits a maximal surface.
\end{defn}
%It is easy to show that  every concave triple could be perturbed to a rational one.

The primary sources of a maximal cap are Donaldson caps.
\begin{lemma}\label{l:DonAreMax}
 A Donaldson hypersurface for $(P,\omega_P)$ is a maximal surface, and hence a Donaldson cap is a maximal cap.
\end{lemma}

\begin{proof}
 It follows directly from Lemma \ref{l:maximalClass}.
\end{proof}

\begin{thm}\label{t:closedDonHyper}
 Any contact 3-manifold admits a maximal cap.
 \end{thm}

\begin{proof}[Proof of Theorem  \ref{t:closedDonHyper}]
By \cite{EtHo02},  there exists  a Stein fillable contact 3-manifold $(Y_2, \xi_2)$ such that $(Y, \xi)$ is Stein cobordant to $(Y_2, \xi_2)$.
Denote the Stein cobordism by $(SC, \tau)$.

Let $(N,\omega_N)$ be a Stein filling of $(Y_2, \xi_2)$.
By \cite{LisMat}, $(N,\omega_N)$ embeds into a minimal surface $X$ of general type. In fact, inspecting their argument,  we see that there is an affine surface $A$ in $X$ such that $N \subset A$
and $X$ is the projective compactification of $A$. The divisor $D:=X \backslash A$ is ample, and
by Hironaka's resolution of singularities we can assume that it is a simple normal crossing divisor.

In particular, we can smooth $D$ out to a smooth symplectic Donaldson hypersurface in $P:=X \backslash N$.
By gluing $P$ with $(SC, \tau)$, we get a maximal cap of $(Y,\xi)$.

\end{proof}

%--------------------------------------------------
%Finally we note that Theorem  \ref{t:closedDonHyper} has the following consequence.

\begin{corr}\label{c:exactCob}

 Any contact 3-manifold  is Stein cobordant to a contact circle bundle.
\end{corr}

\begin{proof}
By the proof of Theorem \ref{t:closedDonHyper}, any contact $3$-manifold $(Y, \xi)$ is Stein cobordant to a Stein fillable one $(Y_2, \xi_2)$, which in turn Stein fillable
to the contact boundary $(Y_3, \xi_3)$ of a regular neighborhood of the compactifying divisor $D$.
We can perturb $D$ to a symplectic divisor $D'$ with symplectic orthogonal intersection points (\cite{Go95}).
By \cite{McL16} (see also \cite{LM}), the contact boundary of a small regular neighborhood $N(D')$ of $D'$ is contactomorphic to $(Y_3, \xi_3)$ and $N(D') \backslash D'$ is equipped
with a Weinstein structure.
We pick a Darboux chart and a compatible integrable complex structure near each intersection points of $D'$ to
resolve the symplectic orthogonal intersection points of $D'$ using the local model $C_{\epsilon}:=\{(z_1,z_2) \in \mathbb{C}^2| z_1z_2=\epsilon\}$, where $\epsilon=0$ corresponds
to the symplectic orthogonal intersection and $\epsilon \neq 0$ corresponds to smoothing.
We denote a choice of smoothing family as $D'_{\epsilon} \subset N(D')$.
By inspecting these local smoothings, one can see that $N(D')\backslash  U(D'_{\epsilon})$ for $\epsilon \neq 0$ can be obtained by adding Weinstein handles to
$N(D') \backslash U(D'_0)$ where $U(D'_{\epsilon})$ is a small open regular neighborhood of $D'_{\epsilon}$ (cf. \cite{Nguyen}).
As a result, $(Y_3, \xi_3)$ is Stein cobordant to $\partial Cl (U(D'_{\epsilon}))$ which is a contact circle bundle.
Here $Cl$ stands for taking closure.
\end{proof}

%Consequently, we have
% $$Kod(Y,\xi)=\min \{Kod(Y_{g,n},\xi_{g,n})| \text{ $(Y,\xi)$ is exact cobordant to $(Y_{g,n},\xi_{g,n})$}\}$$

%By Corollary \ref{c:exactCob}, any contact $3$ manifold $(Y,\xi)$ is exact cobordant to one of the $(Y_{g,n},\xi_{g,n})$ in Lemma \ref{l:sharp}.
%Therefore, we have $Kod(Y,\xi) \le Kod(Y_{g,n},\xi_{g,n})$.
%On the other end, if $Kod(Y,\xi)=-\infty$, we can find a uniruled cap and hence, after possibly perturbation of the symplectic form,
%also a Donaldson hypersurface $D$ in the cap with $c_1(D) >0$.
%Corollary \ref{c:exactCob} implies that $(Y,\xi)$ is exact cobordant to some $(Y_{g,n},\xi_{g,n})$ with $Kod(Y_{g,n},\xi_{g,n})=-\infty$.
%If $Kod(Y,\xi)=0$, a similar reasoning shows that $(Y,\xi)$ is exact cobordant to some $(Y_{g,n},\xi_{g,n})$ with $Kod(Y_{g,n},\xi_{g,n})=0$.
%This completes the proof.
%\end{comment}

\begin{proof}[Proof of Theorem \ref{t:charNumBound}]
 By Theorem \ref{t:closedDonHyper} and Lemma \ref{l:DonAreMax}, we can find a maximal cap.  Therefore, the result follows by Corollary \ref{c:explicitBounds}.

 In particular, if $(Y,\xi)$ admits a polarized cap $(P, \omega_P, D)$ such that the genus of $D$ is greater than $0$,
 then $(2\chi+3\sigma)(N) \ge \min\{c_1(P) \cdot D +2-2g,8-8g\}-(2\chi+3\sigma)(P)$ for any minimal symplectic fillings $N$.
 If $b_2^+(P)>1$, then we have the better bound $(2\chi+3\sigma)(N) \ge c_1(P) \cdot D - (2\chi+3\sigma)(P)$.
\end{proof}

%Following \cite{Op13-2} we further introduce the notion of a singular Donaldson hypersurface for concave symplectic manifolds and establish the existence as well.

%Theorem \ref{t:charNumBound} now follows from Theorem \ref{t:closedDonHyper}, Lemma \ref{l:DonAreMax} and Corollary \ref{c:explicitBounds}.

\begin{eg}\label{e: no upper bound}
Baykur and Van Horn Morris provide infinitely many contact $3$-manifolds $(Y_k,\xi_k)$, each of which has infinitely Stein filling with strictly increasing  $\chi$, strictly decreasing $\sigma$ and strictly increasing $2\chi+3\sigma$ in \cite{BaVHM15}.

\end{eg}

\begin{proof}[Proof of Theorem \ref{t:universalMinimal}]
Let $(Y,\xi)$ be as before and $(P,\omega_P)$ any symplectic cap of $(Y,\xi)$.
By Theorem \ref{t:closedDonHyper}, we can find a Donaldson hypersurface and hence maximal surface $D$ in $(P,\omega_P)$.
Without loss of generality, we can assume that the genus $g$ of $D$ is positive.
Denote the self-intersection number of $D$ as $s$.

We consider a symplectic four tours $(T^4,\omega)$ with product symplectic form.
One can easily find a symplectic surface $D'$ of genus $g$ in $(T^4,\omega)$.
By adjunction, $[D']^2 \ge 0$.
We can perform $[D']^2 +s$ symplectic blow-ups using disjoint balls centered along $D'$ and get a symplectic surface $D"$ of genus $g$ and self-intersection $-s$.
Denote the resulting ambient manifold $(X',\omega')$.
Notice that, by Taubes SW theory, all symplectic exceptional spheres in $(X',\omega')$ intersect $D"$.
We now perform Gompf's symplectic sum surgery between $(P,\omega_P)$ and $(X',\omega')$ along $D$ and $D"$, which results in another symplectic cap $(P',\omega_P')$ of $(Y,\xi)$.

Now, for any minimal symplectic filling $(N,\omega_N)$ of $(Y,\xi)$, the glued symplectic manifold $N \cup P'$
can also be obtained as performing symplectic sum surgery between $N \cup P$ and $X'$.
Since $D$ and $D"$ are maximal in  $N \cup P$ and $X'$, respectively.
The minimality theorem of Usher \cite{Us09} implies that $N \cup P'$ is minimal.
\end{proof}

\subsection{Exact symplectic cobordisms of fillable 3-manifolds}\label{s:ExactCob}

We discuss the analogue of Theorem \ref{t:charNumBound} for exact cobordism of fillable 3-manifolds and propose a weakened version of Stipsicz's conjecture for further study.

\subsubsection{The bound on $2\chi+3\sigma$}

We will prove Corollary  \ref {c:charSelfExact}, in fact the more general  Proposition \ref{l:exactCobBound}, based on the following observation.

\begin{prop}\label{l:exactCobAttach}
 If $W$ is an exact cobordism from $(Y_1, \xi)$ to $(Y_2, \xi_2)$ and $N$ is a strong symplectic filling of $(Y_1, \xi)$ with $k$ exceptional classes,
 then $N \cup_{Y_1} W$ also has exactly $k$ exceptional classes.

 In particular, if $N$ is minimal, then $N \cup_{Y_1} W$ is also minimal.
\end{prop}

\begin{proof}
By Corollary \ref{c:disjointExceptional}, we know that all exceptional classes in $N \cup_{Y_1} W$ are disjoint.
Therefore, it suffices to consider the case where $N$ is minimal by blowing down all the exceptional classes in $N$
and we now want to prove that $N \cup_{Y_1} W$ is minimal.

 Assume the contrary that there is an exceptional class $e \in H_2(N \cup_{Y_1} W)$.
We can cap $N \cup_{Y_1} W$ by $P$ to obtain a closed symplectic manifold $X$.
Since $W$ is exact, the Lefschetz dual of the relative symplectic class $[(\omega_{W \cup_{Y_2} P},\alpha_{W \cup_{Y_2} P})]$ of $W \cup_{Y_2} P$ comes from $H_2(P)$.
We denote the Lefschetz dual as $h$.
Moreover, as $N$ is minimal, we can apply Lemma \ref{l:maximalClass} to the cap $W \cup_{Y_2} P$ and hence $e \cdot h > 0$.
This is a contradiction because $h$ comes from $H_2(P)$ but $e$ comes from $H_2(N \cup_{Y_1} W)$.
\end{proof}

\begin{rmk}
It was pointed out to us by Youlin Li that Etnyre has an argument to show that gluing a minimal strong symplectic filling and a Stein cobordism results in a minimal strong symplectic filling  (See Proposition 2.3 in \cite{DingL}). 
In general, it is not true that gluing a minimal symplectic filling with a minimal symplectic cobordism results in a minimal symplectic filling.
In light of   Proposition \ref{l:exactCobAttach},  exact cobordism is an appropriate kind of cobordism which is good for gluing in the category of
minimal symplectic fillings.

\end{rmk}

For a strongly fillable contact 3-manifold $(Y, \xi)$, let   $n_{Y, \xi}$  be the minimum of $2\chi+3\sigma$ among all of its minimal fillings (which exists by Theorem \ref{t:charNumBound}).

\begin{prop}\label{l:exactCobBound}
 If $(Y_1,\xi_1)$ is strongly symplectic fillable and $(Y_2,\xi_2)$ is another contact 3-manifold, then $2\chi(W)+3\sigma(W) \ge n_{Y_2, \xi_2}-n_{Y_1, \xi_1}$ for all exact cobordism $W$ from $Y_1$ to $Y_2$.

 Moreover, if $Kod(Y_2, \xi_2)=-\infty$ (or $Kod(Y_2, \xi_2)=0$ and $(Y_1,\xi_1)$ is exact fillable), then there are Betti number bounds on $W$.
 \end{prop}

\begin{proof}[Proof of Proposition  \ref{l:exactCobBound} and Corollary \ref {c:charSelfExact}]
 A consequence of Theorem \ref{t:charNumBound}, \ref{t:CY}, \ref{t:uniruled}, Proposition \ref{l:exactCobAttach} and the Novikov signature additivity.
 \end{proof}

\begin{eg}
The examples of Baykur and Van Horn-Morris   mentioned in Example \ref{e: no upper bound}
implies that there is no upper bound of $2\chi+3\sigma$ for exact cobordism from $(S^3,\xi_{std})$ to $(Y_k,\xi_k)$.
\end{eg}

The fillability condition of $(Y_1, \xi_1)$ is important in  Proposition \ref{l:exactCobBound} as seen from  following remark.

\begin{rmk}\label{p:AlmostComplex}
 For any negative integer $n$, there is an exact symplectic cobordism $(W,\omega)$ from $(S^3,\xi_{OT})$ to itself such that $2\chi(W)+3\sigma(W) \le n$.
 Here $\xi_{OT}$ stands for the overtwisted contact structure on $S^3$ in the homotopy class of plane field representing the zero element of $\pi_3(S^2)$.
%\end{prop}

%\begin{proof}
 By \cite{VdeV66} (see also \cite{Gei01}), there is a closed almost complex manifold $(X,J)$ of real dimension four with $2\chi(X)+3\sigma(X) < n$.
 By removing two disjoint small balls from it, we have an almost complex smooth cobordism $W$ from $S^3$ to itself,
 whose homotopy class of $2$-plane fields restricted to either of the two boundary components represent the zero element of $\pi_3(S^2)$.
 This homotopy class supports exactly one overtwisted contact structure on $S^3$ up to isotopy.
 Theorem 1.1 of \cite{EliMur15} (see also the paragraph after it) concludes the remark.
\end{rmk}

\subsubsection{The monoid of exact self-cobordisms}

\begin{defn}
For a closed co-oriented contact $3$-manifold $(Y,\xi)$, the exact (resp. strong, Stein) self-cobordism monoid is the monoid
whose element is exact (resp. strong, Stein) cobordism from $(Y,\xi)$ to itself and multiplication is given by concatenation.
\end{defn}

Since we can always stack elements in the exact/Stein/strong self-cobordism monoid,
if there is any Stein fillable contact $3$-manifold $(Y,\xi)$ such that it has a Stein self-cobordism with non-zero $\chi$ (resp. $\sigma$),
then there is no uniform bound of $\chi$ (resp. $\sigma$) for its Stein fillings.
As a result, any such $(Y,\xi)$ provides a counterexample of Stipsicz's conjecture (see Introduction).
%We weaken Stipsicz's conjecture as follows which seems cannot be directly disproved by the technique in \cite{BaVHM15}.

%\begin{conj}\label{conj}
% For any Stein fillable contact $3$-manifold $(Y,\xi)$, any Stein self-cobordism of $(Y,\xi)$ has zero Euler characteristic and signature.
%In particular, any Stein fillable contact $3$-manifold cannot be obtained from itself via a (non-empty) sequence of Legendrian $(-1)$-surgery.
% \end{conj}

On the other hand, proving that any Stein self-cobordism has zero $\chi$ and $\sigma$ is not an easy question.
If $\chi=\sigma  = 0$ for all Stein self-cobordisms of $(Y,\xi)$, then $(Y,\xi)$ cannot be obtained from itself via a (non-empty) sequence of Legendrian $(-1)$-surgery.
A progress along this line was made by Plamenevskaya who prove that tight lens spaces cannot be obtained from itself via Legendrian $(-1)$-surgery (\cite{Pl12}).
%by proving Betti number bounds on Stein fillings of planar contact $3$ manifolds (cf. \cite{Ka13}) and applying
%Honda's classification of tight contact structures (\cite{Ho00}).
%This implies that any Stein self-cobordism has non-negative 

In fact, we now know that any exact self-cobordism of a strongly  (resp. exact) fillable contact
manifold $(Y,\xi)$ with $Kod(Y,\xi)=-\infty$ (resp. $Kod(Y,\xi)=0$) has  $\chi=\sigma  = 0$,
by Theorem \ref{t:CY}, \ref{t:uniruled} and Proposition \ref{l:exactCobAttach}.
Therefore, it provides a vast amount of Stein fillable contact $3$-manifolds which cannot be obtained from itself via a (non-empty) sequence of Legendrian $(-1)$-surgery.
%Otherwise, we can stack the exact cobordism which will violate Theorem \ref{t:CY} or \ref{t:uniruled} and Lemma \ref{l:exactCobAttach}

Since $\chi=\sigma= 0$ is a strong restriction, the following questions are very natural.

\begin{question}\label{q:trivialMonoid}
Is it true that any exact/Stein self-cobordism monoid of a contact manifold $(Y,\xi)$ with $Kod(Y,\xi)=-\infty,0$ always trivial?
If not, is it possible to classify?
\end{question}

\begin{eg}
Any exact self-cobordism of the tight $(S^3,\xi)$ is smoothly trivial.
This is because any minimal symplectic filling of $(S^3,\xi)$ is diffeomorphic to a ball, any exact self-cobordism is the complement of a ball in a ball and hence smoothly trivial.
\end{eg}

However, counterexamples for first part of Question \ref{q:trivialMonoid} can be easily found as follows.

\begin{prop}\label{p:infMonoid}
 There are infinitely many Stein fillable uniruled contact $3$-manifolds whose Stein self-cobordism monoid is non-trivial.
 Moreover, for each integer $n$, the cardinality of the monoid of a member of this infinite family is larger than $n$.
\end{prop}

\begin{proof}
 Let $(Y,\xi)$ be a connected Stein fillable uniruled contact manifold.
 Then $W_1=[0,1] \times  Y$ is a Stein self-cobordism of $Y$.
 We can attach a Weinstein $1$-handle to $ \{1\} \times Y$ and call the resulting Stein cobordism $W_2$.
 Let the convex end of $W_2$ be $Y'$.
 %Since any two points on $Y$ can be connected by a Legrendrian, from which we obtain a Legrendrian knot in $Y'$.
 We can attach a canceling Weinstein $2$-handle to $W_2$ along a Legrendrian knot and call the resulting Stein cobordism $W_3$.
 %(see P.12-14 of \cite{vk}).
 Let $V_2=Cl(W_2 \backslash W_1)$ and $V_3=Cl(W_3 \backslash W_2)$ be the Stein cobordism from $Y$ to $Y'$ and from $Y'$ to $Y$, respectively.
 By construction, $V_2 \cup_{Y'} V_3$ is smoothly trivial and $b_i(Y')=b_i(Y)+1$ for $i=1,2$, where $b_i$ are the Betti numbers.
 If we stack infinitely many $V_2 \cup_{Y'} V_3$ (or $V_3 \cup_{Y} V_2$) together, we get a smooth manifold diffeomorphic to $ \mathbb{R} \times Y$.
 Since $Y$ is not diffeomorphic to $Y'$, it shows that $V_3 \cup_{Y} V_2$ is not smoothly trivial.
 As a result, $V_3 \cup_{Y} V_2$ is a non-trivial element in the Stein self-cobordism monoid.
 Moreover, it is an idempotent in the monoid.

 Now, it suffices to show that  infinitely many $Y'$ can be obtained in this way and each of them is a Stein fillable uniruled contact $3$-manifolds.
 Since $Y$ is Stein fillable, so is $Y'$.
 On the other hand, $Y'$ being Stein cobordant to $Y$ and $Y$ is uniruled means that $Y'$ is also uniruled.
 Notice that $\#_k (S^1 \times S^2)$ for any $k>0$ can be obtained from $\#_{k-1} (S^1 \times S^2)$ by the above procedure.
 This finishes the proof of the first statement.

 For the second statement, we claim that for each positive integer $l<k$ the Stein self-cobordism $W_{k,l}$ of $\#_k (S^1 \times S^2)$ obtained by
 attaching $k-l$ canceling $2$-handles to $[0,1] \times \#_k (S^1 \times S^2)$ followed by $k-l$ $1$-handles are different for all $l$.
These can be distinguished by the relative homology groups of the cobordisms $W_{k,l}$ relative to the negative boundary.
\end{proof}

We remark that the non-trivial elements in the monoid constructed above are idempotents, which implies that these monoids do not have a group structure.
In contrast, for each $k$, Stein filling of $\#_k(S^1 \times S^2,\xi)$ is unique up to symplectic deformation equivalence.
%\cite{El90}.
This implies that the structure of the monoid is sometimes more complicated than its fillings.

We also remark that there is a natural map from the exact self-cobordism monoid of $(Y,\xi)$
to the monoid of endomorphisms of Heegaard Floer homology/embedded contact homology of  $(Y,\xi)$ modulo automorphism.
Moreover, these endomorphisms preserve the contact invariant.

For example, another way to see that the second statement of Proposition \ref{p:infMonoid} is true is as follows.
The hat version of Heegaard Floer $\widehat{HF}(\#_k (S^1 \times S^2),{\bf s}_0)$ is isomorphic to $H_*(T^k)$, where $T^k$ is the $k$-dimensional torus and ${\bf s}_0$ is the unique $spin^c$ structure such that $c_1({\bf s}_0)=0$ (\cite[Section $9$]{OzSz04}).
  Then, the basic properties of Heegaard Floer homology imply that the induced map $F_{W_{k,l},{\bf s}}: \widehat{HF}(\#_k (S^1 \times S^2),{\bf s}_0) \to \widehat{HF}(\#_k (S^1 \times S^2),{\bf s}_0)$ has rank $2^l$ (actually, the map is a projection to the homology of a sub-torus).
 It then implies the result.
In particular, any non-trivial Stein self-cobordism we constructed does not induce isomorphism on the Heegaard Floer homology.

\begin{question}
Is there any smoothly non-trivial Stein self-cobordism of a Stein fillable contact $3$-manifold which induces isomorphism on the Heegaard Floer homology?
In other words, can Heegaard Floer homology detect the triviality of a Stein self-cobordism  for a Stein fillable $3$-manifold?
\end{question}

%In fact, one can also see that the natural map between the monoids is surjective for $\#_k (S^1 \times S^2,\xi)$.

%\begin{question}
%When is the natural map from the monoid of exact self-cobordisms to the monoid of endomorphisms of Heegaard Floer homology that preserve the contact invariant modulo automorphism surjective?
%\end{question}

%%%%%%%%%%%%%%%%%%%%%%%%%%%%%%%%%%%%%

\bibliography{MakRef}

\bibliographystyle{plain}

%%%%%%%%%%%%%%%%%%%%%%%%%%%%%%%%%

%%%%%%%%%%%%%%%%%%%%%%%%%%%%%%%%%%%%%%%%%%%%

\end{document}